\documentclass[12pt]{amsart}
\usepackage{amsmath,amssymb}
\usepackage{amsfonts}
\usepackage{amsthm}
\usepackage{latexsym}
\usepackage{graphicx}
\def\D{\displaystyle}

\def\pd#1#2{\frac{\partial #1}{\partial #2}}
\def\vv<#1>{\langle#1\rangle}

\def\XXint#1#2{\setbox0=\hbox{$#1{#2}{\int}$}{#2}\kern-.5\wd0 }

\def\XXint#1#2#3{{\setbox0=\hbox{$#1{#2#3}{\int}$}
     \vcenter{\hbox{$#2#3$}}\kern-.5\wd0}}

\def\vv<#1>{\langle#1\rangle}
\newtheorem{thm}{Theorem}[section]

\theoremstyle{definition}

\theoremstyle{remark}

\newtheorem{rem}{Remark}[section]

\numberwithin{equation}{section}

\begin{document}
\title{Time-dependent Sobolev inequality along the  Ricci flow}
\author{Chengjie Yu$^1$}
\address{Department of Mathematics, Shantou University, Shantou, 515063, Guangdong, P.R.China}
\email{cjyu@stu.edu.cn}
\thanks{$^1$ This work is partially supported by a startup funding for research at Shantou University.}
\date{Dec. 2008}

\begin{abstract}
In this article, we get a time-dependent Sobolev inequality along
the Ricci flow in a more general situation than those in Zhang
\cite{Zhang1}, Ye \cite{Ye} and Hsu \cite{Hsu} which also
generalizes the results of them. As an application of the
time-dependent Sobolev inequality, we get a growth of the ratio of
non-collapsing along immortal solutions of Ricci flow.
\end{abstract}

\maketitle\markboth{CHENGJIE YU}{SOBOLEV INEQUALITIES}
\section{Introduction}
Consider the Ricci flow
\begin{equation}
\left\{\begin{array}{l}\frac{d}{dt}g=-2Rc\\g(0)=g_0
\end{array}\right.
\end{equation}
on a closed manifold $M^n$. An important ingredient of Perelman's
proof of geometrization conjecture is the non-collapsing theorem of
Ricci flow which makes sure that we can get a singularity model of
the flow when a singularity exists. In \cite{Zhang1}, Zhang gave an
easier way to prove the non-collapsing theorem of Ricci flow via a
uniform Sobolev inequality along the flow. Unfortunately, there is
a mistake in the proof of Zhang \cite{Zhang1}. Later, Ye \cite{Ye}
corrected the error and Zhang \cite{Zhang2} also corrected the error
by himself.

In \cite{Ye} and \cite{Zhang1}, Ye and Zhang only considered Sobolev inequalities with
$L^2$ right hand side so that the surface case was excluded. In Hsu \cite{Hsu}, she got
uniform Sobolev inequalities with general right hand side so that the surface case was also
included.

In \cite{Ye} and \cite{Hsu}, Ye and Hsu got uniform Sobolev inequalities along the Ricci flow
with the assumption that we are only considering Ricci flow in a finite time interval
or that $\lambda_0(g_0)>0$ where $g_0$ is the initial metric of the Ricci flow and $\lambda_0(g)$
is the first eigenvalue of $-\Delta_{g}+\frac{R(g)}{4}$. Note that, we have the following
evolution inequality of $\lambda_0$ along the Ricci flow (See Kleiner-Lott \cite{KL}),
\begin{equation*}
\frac{d}{dt}\lambda_0\geq \frac{2}{n}\lambda_0^{2}.
\end{equation*}
Therefore, when $\lambda_0(g_0)>0$, the Ricci flow exists only for finite time. So, the second
case is included in the first case that assumes the the time interval is finite.

In this article, we remove the assumption of finite time interval and get a time-dependent Sobolev inequality along
the Ricci flow with Sobolev constants varying when time is varying which generalizs the results
of Zhang \cite{Zhang1}, Ye \cite{Ye} and Hsu \cite{Hsu}. The main result of this article is
as follows.
\begin{thm} Let $g(t)$ with $t\in [0,T)$ be a solution to the Ricci flow on a closed
manifold $M^n$ where $T$ can be $\infty$. Then, there are two positive constants $A$ and $B$ depending only on the initial
metric $g(0)$ and $n$, such that
\begin{equation*}
\begin{split}
&\Big(\int_M|u|^\frac{np}{n-p}dV_t\Big)^\frac{n-p}{np}\\
\leq&\frac{Ae^{Bt}}{n-p}\Bigg(\int_M\Big(\|\nabla u\|^2+\frac{R+\max_MR_{-}(0)+4}{4}u^2\Big)^\frac{p}{2}dV_t\Bigg)^\frac{1}{p}
\end{split}
\end{equation*}
for any $p\in [1,n)$, $t\in [0,T)$ and $u\in C^\infty(M)$.
\end{thm}
Our proof of the main result is mainly the same as in Zhang \cite{Zhang1}, Ye \cite{Ye} and Hsu \cite{Hsu}. Arguments
are mainly contained in Davies \cite{Davies}. First, by the monotunicity of Perelman's W-entropy, we can
get a uniform logarithmic Sobolev inequality. Then, the remaining arguments are somehow standard.
We get the result just by calculating the constants of estimations more carefully via a trick
in Saloff-Coste \cite{S}.
\section{Time-dependent Sobolev inequality and the growth of the ratio of non-collapsing }
Let $M^n$ be a closed manifold with dimension $n$ ($n\geq 2$). $g(t)$  with $t\in [0,T)$ be a solution
to the Ricci flow where $T$ may be $\infty$. By Ye \cite{Ye} and Hsu \cite{Hsu}, we have
the following uniform logarithmic Sobolev inequality.
\begin{thm}\label{thm-logarithmic-sobolev-inequality}
There are two positive constants $A,B$ depending only on the initial metric
$g(0)$ and $n$, such that for any $\sigma>0$ and $t\in [0,T)$
\begin{equation*}
\int_Mv^2\log v^2 dV_t\leq\sigma\int_M\Big(\|\nabla v\|^2+\frac{R}{4}v^2\Big)dV_t-\frac{n}{2}\log\sigma+A(t+\sigma/4)+B
\end{equation*}
for any $v\in C^\infty(M)$ with $\D\int_Mv^2dV_t=1$.
\end{thm}
By the same arguments as in Zhang \cite{Zhang1} and Ye \cite{Ye}
which is mainly contained in Davies \cite{Davies}, we have the
following ultracontractivity of heat kernel assuming a logarithmic
Sobolev inequality. Because our statement is a little different with
that in Ye \cite{Ye}, we also give the proof here.
\begin{thm}\label{thm-ultracontractivity}
Let $(M^n,g)$ be a closed Riemannian manifold,
\begin{equation*}
H=-\Delta+\Psi \ \mbox{and }Q(u)=\int_MuHudV=\int_M(\|\nabla u\|^2+\Psi u^2)dV
\end{equation*}
with $\Psi\in C^\infty(M)$ and $\Psi\geq0$. Let $\sigma_0$ be a positive constant such that for any
$\sigma\in (0,\sigma_0]$, the following logarithmic Sobolev
inequality
\begin{equation*}
\int_Mu^2\log u^2dV\leq\sigma Q(u)+\beta(\sigma)
\end{equation*}
holds for any smooth function $u$ with $\int_Mu^2 dV=1$ where
$\beta(\sigma)$ is an integrable function on $(0,\sigma_0]$. Then,
\begin{equation*}
\|e^{-tH}f\|_\infty\leq \exp\Big({\D\frac{1}{pt}\int_0^t\beta(4\sigma)d\sigma}\Big)\|f\|_p
\end{equation*}
 for any $p\geq 1$, $t\in (0,\sigma_0/4]$ and $f\in C^\infty(M)$.
\end{thm}
\begin{proof}
%We assume that $\Psi\geq 0$, for general case, just replace $\Psi$ by $\Psi+\max\Psi_-$.
The logarithmic Sobolev inequality reads
\begin{equation*}
\int_Mu^p\log u^p\leq \sigma
Q(u^{\frac{p}{2}})+\beta(\sigma)\|u\|_p^{p}+\|u\|_p^{p}\log\|u\|_p^{p}
\end{equation*}
for any nonnegative function $u\in C^\infty(M)$, $p\geq 1$ and
$\sigma\in (0,\sigma_0]$.

Let $u(t)$ be a solution to the Schrodinger equation
\begin{equation*}
\pd{u}{t}=-Hu.
\end{equation*}
Let $p(t)$ be an increasing function to be determined($p(t)\geq 1$).
Compute as follows (assuming that $u$ is nonnegative),
\begin{equation*}
\begin{split}
&\frac{d\log\|u\|_{p(t)}}{dt}\\
=&\frac{d}{dt}\frac{1}{p(t)}\log\int_{M}u^p\\
=&-\frac{p'(t)}{p(t)^2}\log\Big(\int_{M}u^p\Big)+\frac{1}{p(t)\int_{M}u^p}\Big(\int_Mpu^{p-1}u'+\int_Mu^pp'\log u \Big)\\
=&-\frac{p'(t)}{p(t)^2}\log\int_M u^p-\frac{1}{\int_M
u^p}\int_Mu^{p-1}Hu+\frac{p'}{p^2\int_{M}u^p}\int_Mu^{p}{\log u^p}
\end{split}
\end{equation*}
Note that
\begin{equation*}
\begin{split}
\int_M u^{p-1}Hu=&\int_Mu^{p-1}(-\Delta u+\Psi u)\\
=&\int_M\vv<\nabla u^{p-1},\nabla u>+\Psi u^p\\
=&\frac{4(p-1)}{p^2}\int_M\|\nabla u^{p/2}\|^2+\int_M\Psi u^p\\
=&\frac{4(p-1)}{p^2}Q(u^{p/2})+\frac{(p-2)^2}{p^2}\int_M\Psi u^p\\
\geq&\frac{4(p-1)}{p^2}Q(u^{p/2})
\end{split}
\end{equation*}
Therefore,
\begin{equation*}
\begin{split}
&\frac{d\log\|u\|_{p(t)}}{dt}\\
\leq&-\frac{p'(t)}{p(t)^2}\log\int_M u^p-\frac{4(p-1)}{p^2\int_M u^p}\int_MQ(u^{p/2})+\frac{p'}{p^2\int_{M}u^p}\int_Mu^{p}{\log u^p}\\
=&\frac{p'(t)}{p(t)^2\int_Mu^p}\Big(\int_Mu^p\log
u^p-\frac{4(p-1)}{p'}Q(u^{p/2})-\int_Mu^p\log\int_Mu^p\Big).
\end{split}
\end{equation*}
In order to use the logarithmic Sobolev inequality, we need
\begin{equation*}
\frac{4(p(t)-1)}{p'(t)}\leq \sigma_0.
\end{equation*}
Assuming this, we get
\begin{equation*}
\frac{d\log\|u\|_{p(t)}}{dt}\leq
\frac{p'(t)}{p(t)^2}\beta(4(p-1)/p').
\end{equation*}
For each $s>0$, let $p(t)=\frac{ps}{s-t}$, then
\begin{equation*}
\frac{4(p(t)-1)}{p'}=\frac{4(ps/(s-t)-1)}{ps/(s-t)^2}=\frac{4(ps-(s-t))(s-t)}{ps}\leq
4(s-t).
\end{equation*}
So, when $s\in (0,\sigma_0/4]$, we have
\begin{equation*}
\frac{d\log\|u\|_{p(t)}}{dt}\leq \frac{1}{ps}\beta(4(s-t))
\end{equation*}
for any $t\in (0,s]$. Integrating against $t$ on $[0,s]$, we get
\begin{equation*}
\|u(s)\|_\infty\leq
e^{\frac{1}{ps}\int_0^s\beta(4\sigma)d\sigma}\|u(0)\|_p.
\end{equation*}

\end{proof}

From the ultracontractivity of heat kernel, we can get a Sobolev inequality with effective
constants by a combination of a trick in Davies \cite{Davies} and a trick in Saloff-Coste \cite{S}.

\begin{thm}\label{thm-ultra-sobolev}
Let $(M^n,g)$ be a closed Riemannian manifold,
\begin{equation*}
H=-\Delta+\Psi\ \mbox{and}\ \frak q(u)=\|\nabla u\|^2+\Psi u^2
\end{equation*}
with $\Psi\in C^\infty(M)$ and $\Psi\geq1$. Let $p\in [1,n)$, $A\geq 1$ be such that
\begin{equation*}
\|e^{-tH}f\|_\infty\leq At^{-\frac{n}{2p}}\|f\|_p
\end{equation*}
for any $t\in (0,1]$ and $f\in C^\infty(M)$. Then,
\begin{equation*}
\|u\|_{\frac{np}{n-p}}\leq \frac{32An^np}{n-p}\Big(\int_M\frak q(u)^\frac{p}{2}dV\Big)^\frac{1}{p}.
\end{equation*}
for any $u\in C^\infty(M)$.
\end{thm}
\begin{proof} When $t>1$, by maximum principle, we have
\begin{equation*}
\begin{split}
\|e^{-tH}f\|_\infty&\leq e^{-(t-1)}\|e^{-H}f\|_\infty\leq
Ae^{-(t-1)}\|f\|_{p}\leq n^nAt^{-\frac{n}{2p}}\|f\|_{p}
\end{split}
\end{equation*}
for any $f\in C^\infty(M)$. Obviously, the inequality is also true for $t\in (0,1]$.

Note that,
\begin{equation*}
\begin{split}
H^{-1/2}f=&\Gamma\Big(\frac{1}{2}\Big)^{-1}\int_0^{\infty}t^{-1/2}e^{-tH}fdt\\
=&\Gamma(1/2)^{-1}\int_0^{T}t^{-1/2}e^{-tH}fdt+\Gamma(1/2)^{-1}\int_T^{\infty}t^{-1/2}e^{-tH}fdt\\
:=&g+h
\end{split}
\end{equation*}
where $T$ is a positive constant to be determined. Moreover, note that
\begin{equation*}
\|h\|_\infty\leq n^nA\Gamma(1/2)^{-1}\int_T^{\infty} t^{-1/2-n/(2p)}dt\|f\|_p=\frac{2n^nApT^{-\frac{n-p}{2p}}}{(n-p)\Gamma(1/2)}\|f\|_p.
\end{equation*}
For any $\lambda>0$, choose $T>0$ such that
\begin{equation}\label{eqn-T-lambda}
\frac{2n^nApT^{-\frac{n-p}{2p}}}{(n-p)\Gamma(1/2)}\|f\|_p=\lambda/2.
\end{equation}
Then,
\begin{equation*}
\begin{split}
m(\{|H^{-1/2}f|\geq\lambda\})\leq&m(\{|g|\geq\lambda/2\})\leq (\lambda/2)^{-p}\|g\|_p^{p}\\
\leq&(\lambda/2)^{-p}\Gamma(1/2)^{-p}\Big(\int_{0}^{T}\|t^{-1/2}e^{-tH}f\|_pdt\Big)^p\\
\leq& \Gamma(1/2)^{-p}4^p\lambda^{-p}T^\frac{p}{2}\|f\|_p^{p},
\end{split}
\end{equation*}
where we have used the Minkowski inequality and the fact $$\|e^{-tH}f\|_p\leq \|f\|_p.$$

Substituting (\ref{eqn-T-lambda}) into the last equation, we get
\begin{equation*}
\lambda^{\frac{np}{n-p}}m(\{|H^{-1/2}f|\geq\lambda\})\leq 4^\frac{np}{n-p}\Big(\frac{n^nAp}{n-p}\Big)^\frac{p^2}{n-p}\|f\|_p^{\frac{np}{n-p}}.
\end{equation*}
In another word, we get
\begin{equation*}
\lambda^{\frac{np}{n-p}}m(\{|u|\geq\lambda\})\leq 4^\frac{np}{n-p}\Big(\frac{n^nAp}{n-p}\Big)^\frac{p^2}{n-p}\Big(\int_M\frak q(u)^\frac{p}{2}dV\Big)^\frac{n}{n-p}
\end{equation*}
for any $\lambda>0$ and $u\in C^\infty(M)$.

Let $u$ be a nonnegative function. Let $\rho>1$ be a constant to be determined and Let
\begin{equation*}
u_{\rho,k}=(u-\rho^k)_+\wedge(\rho^{k+1}-\rho^k)=\left\{\begin{array}{ll}0&\mbox{if
$u\leq \rho^k$}\\u-\rho^k&\mbox{if $\rho^k<u\leq\rho^{k+1}$}\\
\rho^{k+1}-\rho^k&\mbox{if $u>\rho^{k+1}$}\end{array}\right.,
\end{equation*}
where $k$ is any integer. Then,
\begin{equation*}
\begin{split}
&(\rho^{k+1}-\rho^k)^\frac{np}{n-p}m(\{u\geq\rho^{k+1}\})\\
=&(\rho^{k+1}-\rho^k)^\frac{np}{n-p}m(\{u_{\rho,k}\geq\rho^{k+1}-\rho^k\})\\
\leq&4^\frac{np}{n-p}\Big(\frac{n^nAp}{n-p}\Big)^\frac{p^2}{n-p}\Big(\int_M\frak q(u_{\rho,k})^\frac{p}{2}dV\Big)^\frac{n}{n-p}\\
\end{split}
\end{equation*}
So,
\begin{equation*}
\begin{split}
&\sum_{k=-\infty}^\infty(\rho^{k+1}-\rho^k)^\frac{np}{n-p}m(\{u\geq\rho^{k+1}\})\\
\leq&4^\frac{np}{n-p}\Big(\frac{n^nAp}{n-p}\Big)^\frac{p^2}{n-p}\sum_{k=-\infty}^{\infty}\Big(\int_M\frak q(u_{\rho,k})^\frac{p}{2}dV\Big)^\frac{n}{n-p}\\
\leq&4^\frac{np}{n-p}\Big(\frac{n^nAp}{n-p}\Big)^\frac{p^2}{n-p}\Big(\sum_{k=-\infty}^{\infty}\int_{M}\frak q(u_{\rho,k})^\frac{p}{2}dV\Big)^\frac{n}{n-p}\\
\leq&4^\frac{np}{n-p}\Big(\frac{n^nAp}{n-p}\Big)^\frac{p^2}{n-p}\Big(\sqrt{2}^p\sum_{k=-\infty}^{\infty}\int_{M}(\|\nabla u_{\rho,k}\|^p+\Psi^{\frac{p}{2}}u_{\rho,k}^p)dV\Big)^\frac{n}{n-p}\\
\leq&4^\frac{np}{n-p}\Big(\frac{n^nAp}{n-p}\Big)^\frac{p^2}{n-p}\Big(\sqrt{2}^p\int_M(\|\nabla u\|^p+\Psi^\frac{p}{2}u^{p})dV\Big)^\frac{n}{n-p}\\
\leq&4^\frac{np}{n-p}\Big(\frac{n^nAp}{n-p}\Big)^\frac{p^2}{n-p}\Big(\sqrt{2}^p\cdot\sqrt 2\int_M\frak q(u)^\frac{p}{2}dV\Big)^\frac{n}{n-p}\\
\leq&8^\frac{np}{n-p}\Big(\frac{n^nAp}{n-p}\Big)^\frac{p^2}{n-p}\Big(\int_M\frak q(u)^\frac{p}{2}dV\Big)^\frac{n}{n-p}\\
\end{split}
\end{equation*}
where we have used the inequalities
\begin{equation*}
\frac{1}{\sqrt
2}(a^\frac{p}{2}+b^\frac{p}{2})\leq(a+b)^{\frac{p}{2}}\leq
(\sqrt2)^{p}(a^\frac{p}{2}+b^\frac{p}{2})
\end{equation*}
for any $a,b\geq 0$,
\begin{equation*}
\sum_{k=-\infty}^{\infty}\int_{M}\|\nabla u_{\rho,k}\|^pdV=\int_M\|\nabla u\|^pdV,
\end{equation*}
and
\begin{equation*}
\begin{split}
&\sum_{k=-\infty}^{\infty}\int_{M}\Psi^{\frac{p}{2}}u_{\rho,k}^pdV\\
=&\sum_{k=-\infty}^{\infty}\int_{\{\rho^k<u\leq\rho^{k+1}\}}\Psi^{\frac{p}{2}}(u-\rho^k)^pdV+\sum_{k=-\infty}^{\infty}\int_{\{u>\rho^{k+1}\}}\Psi^{\frac{p}{2}}(\rho^{k+1}-\rho^k)^pdV\\
\leq&\sum_{k=-\infty}^{\infty}\int_{\{\rho^k<u\leq\rho^{k+1}\}}\Psi^{\frac{p}{2}}(u-\rho^k)^pdV+\sum_{k=-\infty}^{\infty}\int_{\{u>\rho^{k+1}\}}\Psi^{\frac{p}{2}}(\rho^{(k+1)p}-\rho^{kp})dV\\
=&\sum_{k=-\infty}^{\infty}\int_{\{\rho^k<u\leq\rho^{k+1}\}}\Psi^{\frac{p}{2}}(u-\rho^k)^pdV+\sum_{k=-\infty}^{\infty}\int_{\{u>\rho^{k}\}}\Psi^{\frac{p}{2}}\rho^{kp}dV-\sum_{k=-\infty}^{\infty}\int_{\{u>\rho^{k+1}\}}\Psi^{\frac{p}{2}}\rho^{kp}dV\\
=&\sum_{k=-\infty}^{\infty}\int_{\{\rho^k<u\leq\rho^{k+1}\}}\Psi^{\frac{p}{2}}\big((u-\rho^k)^p+\rho^{kp}\big)dV\\
\leq&\sum_{k=-\infty}^{\infty}\int_{\{\rho^k<u\leq\rho^{k+1}\}}\Psi^{\frac{p}{2}}u^pdV=\int_M\Psi^\frac{p}{2}u^pdV.\\
\end{split}
\end{equation*}
On the other hand,
\begin{equation*}
\begin{split}
&\sum_{k=-\infty}^\infty(\rho^{k+1}-\rho^k)^\frac{np}{n-p}m(\{u\geq\rho^{k+1}\})\\
=&\sum_{k=-\infty}^\infty(\rho-1)^\frac{np}{n-p}\rho^\frac{knp}{n-p}m(\{u\geq\rho^{k+1}\})\\
\geq&(\rho-1)^\frac{np}{n-p}\sum_{k=-\infty}^\infty\rho^\frac{knp}{n-p}\rho^{-\frac{(np+p-n)(k+2)}{n-p}}(\rho^{k+2}-\rho^{k+1})^{-1}\int_{\rho^{k+1}}^{\rho^{k+2}}t^{\frac{np+p-n}{n-p}}m(\{u\geq t\})dt\\
=&(\rho-1)^\frac{np+p-n}{n-p}\rho^{-\frac{2np+p-n}{n-p}}\int_0^{\infty}t^{^\frac{np+p-n}{n-p}}m(\{u\geq t\})dt\\
=&\frac{n-p}{np}(\rho-1)^\frac{np+p-n}{n-p}\rho^{-\frac{2np+p-n}{n-p}}\int_0^{\infty}t^{^\frac{np+p-n}{n-p}}m(\{u\geq t\})dt\\
=&\frac{n-p}{np}(\rho-1)^\frac{np+p-n}{n-p}\rho^{-\frac{2np+p-n}{n-p}}\int_Mu^{\frac{np}{n-p}}dV
\end{split}
\end{equation*}
Letting $\rho=2$ and combining the last two inequalities, we get
\begin{equation*}
\begin{split}
\int_M u^\frac{np}{n-p}dV\leq&\frac{np}{n-p}\times2^{\frac{2np+p-n}{n-p}}\times 8^\frac{np}{n-p}\Big(\frac{n^nAp}{n-p}\Big)^\frac{p^2}{n-p}\Big(\int_{M}\frak q(u)^\frac{p}{2}dV\Big)^\frac{n}{n-p}\\
\leq&n^{\frac{n^2p}{n-p}}\Big(\frac{32Ap}{n-p}\Big)^\frac{np}{n-p}\Big(\int_{M}\frak q(u)^\frac{p}{2}dV\Big)^\frac{n}{n-p}\\
\end{split}
\end{equation*}
Hence,
\begin{equation*}
\|u\|_{\frac{np}{n-p}}\leq
\frac{32An^np}{n-p}\Big(\int_M\frak{q}(u)^\frac{p}{2}dV\Big)^\frac{1}{p}.
\end{equation*}
\end{proof}

Combining the last three theorems, we get the following time-dependent Sobolev inequality along the Ricci flow.
\begin{thm} Let $g(t)$ with $t\in [0,T)$ be a solution to the Ricci flow on a closed
manifold $M^n$ where $T$ can be $\infty$. Then, there are two positive constants $A$ and $B$ depending only on the initial
metric $g(0)$ and $n$, such that
\begin{equation*}
\begin{split}
&\Big(\int_M|u|^\frac{np}{n-p}dV_t\Big)^\frac{n-p}{np}\\
\leq&\frac{Ae^{Bt}}{n-p}\Bigg(\int_M\Big(\|\nabla u\|^2+\frac{R+\max_MR_-(0)+4}{4}u^2\Big)^\frac{p}{2}dV_t\Bigg)^\frac{1}{p}
\end{split}
\end{equation*}
for any $p\in [1,n)$, $t\in [0,T)$ and $u\in C^\infty(M)$.
\end{thm}
\begin{proof}
By Theorem \ref{thm-logarithmic-sobolev-inequality}, we have two positive constants $C_1$ and $C_2$
depending only on the initial metric $g(0)$, such that for any $\sigma\in (0,4]$ and $t\in [0,T)$
\begin{equation*}
\int_Mv^2\log v^2 dV_t\leq\sigma\int_M\Big(\|\nabla v\|^2+\frac{R}{4}v^2\Big)dV_t-\frac{n}{2}\log\sigma+C_1t+C_2
\end{equation*}
for any $v\in C^\infty(M)$ with $\D\int_Mv^2dV_t=1$.

For each $s\in [0,T)$, let
\begin{equation*}
H_s=-\Delta_s+\frac{R(s)+\max_M R_-(s)}{4}+1,
\end{equation*}
\begin{equation*}
\frak{q}_s(u)=\vv<\nabla^su,\nabla^su>_s+\Big(\frac{R(s)+\max_M R_-(s)}{4}+1\Big)u^2
\end{equation*}
and
\begin{equation*}
Q(u)=\int_M\frak{q}_s(u)dV_s.
\end{equation*}
Then,
\begin{equation*}
\int_Mu^2\log u^2 dV_s\leq\sigma Q_s(u)-\frac{n}{2}\log\sigma+C_1s+C_2
\end{equation*}
for any $\sigma\in (0,4]$ and for any $u\in C^\infty(M)$ with $\int_Mu^2dV_s=1$.

By Theorem \ref{thm-ultracontractivity}, we have
\begin{equation*}
\|e^{-tH_s}f\|_\infty\leq e^{C_1s+C_3}t^{-\frac{n}{2p}}\|f\|_p
\end{equation*}
for any $f\in C^\infty(M)$, $t\in (0,1]$ and $p\geq 1$, where $C_3=C_2+\frac{n}{2}$.

Finally, by Theorem \ref{thm-ultra-sobolev}, we have
\begin{equation*}
\|u\|_{\frac{np}{n-p}}\leq\frac{32n^np\exp(C_1s+C_3)}{n-p}\|\sqrt{\frak{q}_s(u)}\|_p.
\end{equation*}
on $(M^n,g(s))$, for any $u\in C^\infty(M)$ and $p\in [1,n)$.

Noting that $\max R_-(t)$ decreases along the Ricci flow, we get the following time-dependent Sobolev inequality
along the Ricci flow:
\begin{equation*}
\begin{split}
&\Big(\int_M|u|^\frac{np}{n-p}dV_t\Big)^\frac{n-p}{np}\\
\leq&\frac{16n^np\exp(C_1t+C_3)}{n-p}\Bigg(\int_M\Big(\|\nabla u\|^2+\frac{R+\max_MR_-(0)+4}{4}u^2\Big)^\frac{p}{2}dV_t\Bigg)^\frac{1}{p}
\end{split}
\end{equation*}
for any $t\in [0,T)$ and $u\in C^\infty(M)$.
\end{proof}
\begin{rem}
When $T$ is finite, we get the uniform Sobolev inequalities in Zhang \cite{Zhang1}, Ye \cite{Ye} and
Hsu \cite{Hsu}.
\end{rem}

By a standard arguments via iteration as in Zhang \cite{Zhang1}, Ye \cite{Ye} and
Hsu \cite{Hsu}, we get the following growth of the ratio non-collapsing along Ricci flow.

\begin{thm}
Let $g(t)$ with $t\in [0,T)$ be a solution to the Ricci flow on a closed
manifold $M^n$ where $T$ can be $\infty$. Then, there are two positive constants $\kappa$ and $A$ depending only on the initial
metric $g(0)$ and $n$, such that
\begin{equation*}
V_x(r,t)\geq \kappa e^{-At}r^n
\end{equation*}
whenever $r\leq 1$ and $r^2R(t)\leq 1$ on $B_x(r,t)$, where
$B_x(r,t)$ is the geodesic ball of radius $r$ centered $x$ with
respect to $g(t)$ and $V_x(r,t)$ is the volume of $B_x(r,t)$ with
respect to $g(t)$.
\end{thm}

\end{document}